\input ./style/arxiv-general.cfg
\documentclass[aop,MSNbibl,seceqn,dvips]{arximspdf}
\makeatletter
   \@ifpackageloaded{graphicx}{}{\usepackage{graphicx}}
\makeatother

\doi{10.1214/14-AOP981} 
\volume{44}
\issue{1}
\pubyear{2016}
\firstpage{544}
\lastpage{566}
\docsubty{FLA}

\makeatletter
\newcommand{\rrvert}{\vert}
\newcommand{\llvert}{\vert}
\newtheorem{theorem}{Theorem}[section]
\newtheorem{lemma}[theorem]{Lemma}

\newproclaim{definition}[theorem]{Definition}
\newproclaim{example}[theorem]{Example}
\newproclaim{assumption}[theorem]{Assumption}

\newproclaim{remark}{Remark}[section]

\newcommand{\E}{\mathbb{E}}
\newcommand{\F}{\mathcal{F}}
\newcommand{\G}{\mathcal{G}}
\def\h{{\mathfrak h}}
\def\m{{\mathfrak m}}
\newcommand{\R}{\mathbb{R}}
\def\L{\mathcal{L}}

\def\Osc{\mathop{\operatorname{Osc}}}
\def\error{\mathcal E}
\def\tr{\mathop{\operatorname{tr}}}
\def\so{{\mathfrak{so}}}
\def\Lip{\operatorname{Lip}}
\makeatother

\begin{document}
\begin{frontmatter}

\title{Random perturbation to the geodesic equation\thanksref{T1}}
\runtitle{Perturbation to geodesic equation}
\thankstext{T1}{Supported by EPSRC Grant EP/E058124/1.}

\begin{aug}
\author[A]{\fnms{Xue-Mei}~\snm{Li}\corref{}\ead[label=e1]{xue-mei.li@warwick.ac.uk}}
\runauthor{X.-M. Li}
\affiliation{University of Warwick}

\address[A]{Mathematics Institute\\
University of Warwick\\
Coventry CV4 7AL\\
United Kingdom\\
\printead{e1}}
\end{aug}

\received{\smonth{8} \syear{2014}}
\revised{\smonth{10} \syear{2014}}

%
\begin{abstract}
We study random ``perturbation'' to the geodesic equation. The geodesic
equation is identified with
a canonical differential equation on the orthonormal frame bundle
driven by a horizontal
vector field of norm $1$.
We prove that the projections of the solutions to the perturbed
equations, converge,
after suitable rescaling, to a Brownian motion scaled by ${8\over
n(n-1)}$\vspace*{1pt} where $n$ is the dimension of the state space.
Their horizontal lifts to the orthonormal frame bundle converge also,
to a scaled horizontal Brownian motion.
\end{abstract}

%
\end{frontmatter}

\section{Introduction}

Let $M$ be a complete smooth Riemannian manifold of dimension $n$ and
$T_xM$ its tangent space at $x\in M$.
Let $\mathit{OM}$ denote the space of orthonormal frames on $M$ and $\pi$ the
projection that takes an
orthonormal frame $u\dvtx \R^n\to T_xM$ to the point $x$ in $M$. Let
$T_u\pi$ denote its differential at $u$.
For $e\in\R^n$, let $H_u(e)$ be the basic horizontal vector field on
$\mathit{OM}$ such that $T_u\pi(H_u(e))=u(e)$,
that is, $H_u(e)$ is the horizontal lift of the tangent vector $u(e)$
through $u$.
If $\{e_i\}$ is an orthonormal basis of $\R^n$, the second-order
differential operator
$\Delta_H=\sum_{i=1}^nL_{H(e_i)}L_{H(e_i)}$ is the Horizontal Laplacian.
Let $\{w_t^i, 1\le i\le n\}$ be a family of real valued independent
Brownian motions. The solution $(u_t, t<\zeta)$,
to the following semi-elliptic stochastic differential equation (SDE),
$du_t=\sum_{i=1}^n H_{u_t} (e_i)\circ dw_t^i$, is a Markov
process with infinitesimal generator ${1\over2}\Delta_H$ and lifetime
$\zeta$. We denote by
$\circ$ Stratonovich integration. The solutions are known as horizontal
Brownian motions.
It is well known that a horizontal Brownian motion projects to a
Brownian motion on $M$.
We recall that a Brownian motion on $M$ is a sample continuous strong
Markov process with generator
${1\over2}\Delta$ where $\Delta$ is the Laplace--Beltrami operator.
This construction of Brownian motions on a Riemannian
manifold is canonical and has fundamental applications in analysis on
path spaces.

For $e_0\in\R^n$, the horizontal vector field $H(e_0)$ does not
project to a vector field on $M$.
It, however, induces a vector field $X$ on $TM$ which is a geodesic
spray. If $(u_t^{e_0})$ is the solution to the first-order differential
equation
\[
\dot u(t)=H_{u(t)}(e_0),\qquad u(0)=u_0,
\]
then $\pi(u^{e_0}_t) $ is the geodesic on $M$ with initial velocity
$u_0(e_0)$ and initial value $\pi(u_0)$.

Let $N={n(n-1)\over2}$ and let $\so(n)$ be the space of skew
symmetric matrices in dimension $n$. It is the Lie algebra of the
orthogonal group $O(n)$.
For $A\in\so(n)$, we denote by $A^*$ the fundamental vertical vector
field on $\mathit{OM}$ determined by right actions of
the exponentials of $tA$; see (\ref{vv}) below.
If $X$ is a vector field, we denote by $L_X$ Lie differentiation in
the direction of $X$. Let us fix a time $T>0$.
Let $\rho$ be the Riemannian distance function on $M$, $\nabla$
the Levi--Civita connection and $\Delta$ the Laplace--Beltrami operator.
Let $\varepsilon$ a positive number. Our main theorems concern the
convergence, as $\varepsilon$ approaches zero,
of the ``horizontal part'' of the solutions to a family of stochastic
differential equations with parameter $\varepsilon$.
The definitions for the horizontal and vertical vector fields and for
the horizontal lift of a curve are given in Section~\ref{section2}.
Let $e_0 $ be a unit vector in $\R^n$.

\begin{theorem}
\label{thm1}
Let $M$ be a complete Riemannian manifold of dimension \mbox{$n>1$} and of
positive injectivity radius.
Suppose that there are positive numbers $C$ and $a$ such that
$\sup_{\rho(x,y) \le a} |\nabla\, d \rho|(x,y) \le C$.
Let $x_0 \in M$ and $u_0\in\pi^{-1}(x_0)$.
Let $\bar A\in\so(n)$ and $\{A_1,\ldots, A_N\}$ be an orthonormal basis of $\so(n)$.
Let $(u_t^\varepsilon, 0\le t \le T)$ be the solution to the SDE
%
\begin{equation}
\label{ou-1}\cases{ \displaystyle du_t^\varepsilon= H_{u_t^\varepsilon}(e_0)\,dt
+{1\over\sqrt\varepsilon} \sum_{k=1}^NA_k^*
\bigl(u_t^\varepsilon\bigr)\circ dw_t^k+
\bar A^* \bigl(u_t^\varepsilon\bigr) \,dt, \vspace*{2pt}
\cr
u_0^\varepsilon=u_0.}
\end{equation}
Let $x_t^\varepsilon=\pi(u_t^\varepsilon)$ and let $(\tilde x_t^\varepsilon,
0\le t\le T)$ be the horizontal lift of ($x_t^\varepsilon, 0\le t\le T)$
to $\mathit{OM}$ through $u_0$. Then the following statements hold:
\begin{longlist}[(1)]
\item[(1)] The SDE does not explode.
\item[(2)] The processes $(x_{t/\varepsilon}^\varepsilon, 0\le t\le T)$
and $(\tilde x_{t/\varepsilon}^\varepsilon, 0\le t\le T) $ converge in
law, as $\varepsilon\to0$.
\item[(3)] The limiting law of $(x_{t/\varepsilon}^\varepsilon, 0\le
t\le T) $ is independent of $e_0$. It is
a scaled Brownian motion with generator ${4\over n(n-1)}\Delta$.
The limiting law of $(\tilde x_{t/\varepsilon}^\varepsilon, 0\le t\le
T)$ is that associated to the generator ${4\over n(n-1)}\Delta_H $.
\end{longlist}
\end{theorem}

If ``$\varepsilon=\infty$'' and $\bar A=0$, the SDE (\ref{ou-1}) reduces to
the first-order differential equation $\dot u(t)=H_{u(t)}(e_0)$ whose solutions
are geodesics. If ``$\varepsilon=0$'', the SDE ``reduces'' to the ``vertical SDE'',
$du_t^\varepsilon= {1\over\sqrt\varepsilon} \sum_{k=1}^NA_k^*(u_t^\varepsilon
)\circ dw_t^k$. This vertical equation
does not have a meaning for $\varepsilon=0$, nevertheless the ``vertical SDE''
has a first integral $\pi\dvtx  \mathit{OM}\to M$, that is, $\pi(u_t^\varepsilon)=\pi
(u_0^\varepsilon)$.
By a preliminary multi-scale analysis,
we see that $\pi(u_t^\varepsilon)$ varies slowly with $\varepsilon$ and there
is a visible effective motion in the time interval $[0, {1\over
\varepsilon}]$.
The first integral $\pi$ is not a real-valued function. It is a
function from a manifold to a manifold and the
slow variables $\{(x_t^\varepsilon), \varepsilon>0\}$ are not Markov
processes. Before further discussions on conservation laws related to
the SDEs,
we remark the following features:
(1) the slow motion solves a first-order differential equation, (2)
the ``fast motion'' on $\mathit{OM}$ is not elliptic, (3)
the limiting process is semi-elliptic. Another feature of Theorem~\ref
{thm1} is that the pair of the intertwined family of stochastic
processes $(x_{t/\varepsilon}^\varepsilon, \tilde x_{t/\varepsilon
}^\varepsilon)$ converge.
We will explore~(3) in a forthcoming article on homogeneous manifolds.
For now, the following observation indicates a potential
application of (3): the stochastic area of
two linear Brownian motions $\{w_t^1, w_t^2\}$ is the principal part of the
horizontal lift of the two-dimensional Brownian motion $(w_t^1, w_t^2)$
to the three-dimensional Heisenberg group. We remark also that the
first-order horizontal geodesic equation on the orthonormal frame
bundle corresponds
a second-order differential equation on the manifold, which explains
the unusual scaling in~(\ref{ou-1}).

There have been many studies of limit theorems whose geometric settings
or scalings or methodologies relate that in this article.
For example, our philosophy agrees with that in Bismut \cite
{Bismut-hypoelliptic-Laplacian} where the equation
$\ddot x={1\over T}(-\dot x+\dot w)$ interpolates between classical
Brownian motion $(T\to0)$ and the geodesic flow ($T\to\infty$).
In Ikeda \cite{Ikeda87} and Ikeda and Ochi \cite{Ikeda-Ochi}, the authors
studied limit theorems for
line integrals of the form $\int_0^t \phi(dx_s)$, where $\phi$ is a
differential form and $(x_s)$ is a suitable process such as a Brownian
motion. In Manabe and Ochi \cite{Manabe-Ochi} the authors obtained central
limit theorems for line integrals along geodesic flows. One of their
tools is symbolic representations of geodesic flows. 
Another related work can be found in Pinsky \cite{Pinsy-parallel},
where a piecewise geodesic with a Poisson-type switching mechanism is
shown to converge to the horizontal Brownian motion.
We also note that geodesic flows perturbed by vertical Brownian motions
were considered by Franchi and Le Jan \cite{Franchi-LeJan},
in the context of relativistic diffusions.

The conclusion of (\ref{thm1}) is consistent with the following central
limit theorems for geodesic flows.
Let $M$ be a manifold of constant negative curvature and of finite volume.
Let $(\gamma_t(x,v))$ denote the geodesic with initial value $(x,v)$ in
the unit tangent bundle $\mathit{STM}$ and let $\theta_t(v)= (\gamma_t(x,v),
\dot\gamma_t(x,v))$, a stochastic process on $\mathit{STM}$.
Let $f$ be a bounded measurable function on $\mathit{STM}$ with the property
that it is centered with respect to the normalized Liouville
measure $\m$. Then there is a number $\sigma$ with the property that
\[
\lim_{t\to\infty} \m \biggl\{\xi\dvtx {\int_0^t f(\theta_s (\xi))
\,ds\over
\sigma\sqrt t} \le a
\biggr\} ={1\over\sqrt{2\pi}}\int_{-\infty}^a
e^{-y^2/2} \,dy.
\]
See Sinai \cite{Sinai60}, Ratner \cite{Ratner73}; see Guivarch and Le Jan
\cite{Guivarch-LeJan} and Enriquez, Franchi and Le Jan \cite
{Enriquez-Franchi-LeJan} for further developments.
See also Helland \cite{Helland} and Kipnis and Varadhan
\cite{Kipnis-Varadhan}. These results
exploit the chaotic nature of the deterministic dynamical system on
manifolds of negative curvature.

In the homogenisation literature, the following works are particularly relevant:
Khasminskii \cite{Khasminskii63,Hasminskii68}, Nelson \cite{Nelson67},
Borodin and Freidlin \cite{Borodin-Freidlin95}, Freidlin and Wentzell~\cite
{Freidlin-Wentzell98}
and Bensoussan, Lions and Papanicolaou \cite{{Bensoussan-Lions-Papanicolaou11}}.
We note in particular Theorem~2.1 in \cite{Borodin-Freidlin95} which
deals with the convergence
of path integrals of a suitable function along a family of ergodic
Markov processes.
In this article, such integrals are better understood as integrals of
differential 1-forms along random paths.
Finally, we mention the following work: Li \cite{Li-averaging} for
averaging of integrable systems
and Ruffino and Gonzales Gargate \cite{Gargate-Ruffino} for averaging on foliated
manifolds.
See also \cite{Li-OM} for an earlier work on the orthonormal frame bundle.
We also refer to Dowell \cite{Dowell} for a scaling limit of
Ornstein--Uhlenbeck type.

\textit{Open question.}
The local uniform bound on $\nabla\, d\rho$ is only used in Lemma~\ref
{lem2} for the proof of tightness.
This bound can be weakened, for example, replaced by a local uniform
control over
the rate of growth of the norms of ${\nabla\, d\rho\over\rho}$ and
${\nabla\,\rho\over\rho}$.
We remark that Brownian motion constructed in Theorem~\ref{thm1} is
automatically complete.
The conditions in Theorem~\ref{thm1} appear to be related to the
uniform cover criterion on stochastic
completeness and could be studied in connection with that in Li \cite
{Li-infinity}.
Also, much of the work in this article is valid for a connection
$\nabla
$ with torsion,
the horizontal tangent bundle and $\Delta_H$
will then be induced by this connection with torsion. The effect of the torsion
will generally lead to an additional drift to the Brownian motion
downstairs. In this case the
geodesic completeness of the manifold $M$
may no longer be equivalent to the metric completeness of $(M, \rho)$.

\section{Preliminaries}
\label{section2}
Given a Riemannian metric on $M$,
an orthonormal frame $u=\{u_1, \ldots, u_n\}$ is an ordered basis of
$T_xM$ that is orthonormal.
We denote by $\mathit{OM}$ the set of all orthonormal frames on $M$ and
$\pi$ the map that takes the frame $u$ to the point $x\in M$.
Let $\pi^{-1}(x)=\{ u\in \mathit{OM}\dvtx  \pi(u)=x \}$.
If $(O, x)$ is a coordinate system on $M$, $u_i=\sum_j u_i^j {\partial
\over\partial x_j} |_{x}$.
This gives a coordinate map on $\mathit{OM}$.\vspace*{-2pt} The map $(x, u_i^j)$ is a
homeomorphism from $\pi^{-1}(O)$ to
$(x(O), O(n))$.
If we identify a frame $u$ with the transformation $u\dvtx \R^n \to T_xM$,
then $\mathit{OM}$ is a principal bundle with fibre $O(n)$ and group $G$, acting
on the right. We adopt the notation $ue=u(e)$.
For $g\in O(n)$ let $R_g$ denote right multiplication on $O(n)$ and the
right action of $O(n)$ on $\mathit{OM}$.
For $A, B\in\so(n)$ let $\langle A, B\rangle =\tr AB^T$.

A tangent vector $v$ in $\mathit{OM}$ is vertical if $T\pi(v)=0$ where $T\pi$
denotes the differential of $\pi$.
If $A$ belongs to the Lie algebra $\so(n)$, we denote by $\exp(tA)$ the
exponential map. If $u$ is a frame, the composition $u\exp(tA)$
is again a frame in the same fibre. We define the fundamental vertical
vector fields associated to $A$ by $A^*$,
%
\begin{equation}
\label{vv} A^*(u)={d\over dt} \bigg|_{t=0} u\exp(tA).
\end{equation}

By a linear connection on the principal bundle $\mathit{OM}$,
we mean a splitting of the tangent bundle $T\mathit{OM}$ with the following properties:
(1) $T_u\mathit{OM}=HT_u\mathit{OM}\oplus VT_u\mathit{OM}$
(2) $(R_a)_*H_uT\mathit{OM}=H_{ua}T\mathit{OM}$ for all $u\in \mathit{OM}$ and $a\in G$.
The spaces $HT_u\mathit{OM}$ and $VT_u\mathit{OM}$ are, respectively, the horizontal
tangent spaces and the vertical
tangent spaces. 
We will introduce a metric on $\mathit{OM}$ such that $\pi$ is an isometry
between $H_uT\mathit{OM}$ and $T_{\pi(u)}M$ and such that
$H_uT\mathit{OM}$ and $VT_u\mathit{OM}$ are orthogonal. The metric on $\so(n)$ is the
bi-invariant metric introduced earlier.
We will restrict our attention to the Levi--Civita connection.

Let $\h_u(v)$ denote the horizontal lift of $v\in T_xM$ through $u\in
\pi^{-1}(x)$.
To each $e\in\R^n$ we denote $H_u(e)=\h_u(u e)$ the basic vector
field. Later, we also use $H_ue$ for $H_u(e)$.
If $\{e_1, \ldots, e_n\}$ is an orthonormal basis of $\R^n$, then $\{
H_u(e_1), \ldots, H_u(e_n)\}$ is an
orthonormal basis for the horizontal tangent space $HT_u\mathit{OM}$.

A piecewise $C^1$ curve $\gamma$ on $\mathit{OM}$ is horizontal if the one-sided
derivatives $\dot\gamma(\pm)$ are
horizontal for all $t$. If $c$ is a $C^1$ curve on $M$, there is a
horizontal curve $\tilde c$ on $\mathit{OM}$
such that $\tilde c$ covers $c$, that is, $\pi(\tilde c(t))=c(t)$. In
fact, $\tilde c(t)$
is the family of orthonormal frames along $c$ that are obtained by
parallel transporting the
frame $\tilde c(0)$. We say that $\tilde c$ is a horizontal lift of
$c$. The map
$\tilde c(t)(\tilde c(0))^{-1}\dvtx T_{c(0)}M\to T_{c(t)}M$ is the
parallel translation along the curve $c(t)$.
In a coordinate chart $(O,x)$, the principal part of $\tilde c(t)$ is a
$n\times n$ matrix whose column vectors $\{\tilde c_1(t), \ldots,
\tilde
c_n(t)\}$ form a frame.
In components, write $\tilde c_l(t)=(\tilde c_l^1(t), \ldots, \tilde
c_l^n(t))^T$. Then
\[
{\partial\tilde c_l^k(t)\over\partial t}+\sum_{i=1, j=1}^n
{\partial
c^i(t)\over\partial t}\Gamma^k_{ij}\bigl(c(t)\bigr) \tilde
c_l^j(t)=0.
\]
Take $c(t)=(0, \ldots, t, \ldots, 0)$, where the nonzero entry is in the
$i$th-place. We obtain the principal part of
the horizontal lift of ${\partial\over\partial x_i}$ through $u=
\tilde c(0)=(u_l^j)$:
\[
\biggl(\h_{\tilde c(0)} \biggl({\partial\over\partial x_i}\biggr)
\biggr)_l= \biggl({\partial\tilde c\over\partial t}(0) \biggr)_l
=- \Biggl(\sum_{j}\Gamma^1_{ij}
u_l^j, \ldots, \sum_{j=1}^n
\Gamma ^n_{ij}u_l^j
\Biggr)^T.
\]
Denote by $A_i$ the matrix whose element at the $(b,l)$ position is
$\sum_{j}\Gamma^b_{ij}u_l^j$. Then $A_i$ is the principal part of
$H_u({\partial\over\partial x_i})$
and the horizontal space at $u$ is spanned by the basis $\{ ({\partial
\over\partial x_i}, A_i)\}$.

A basic object we use in our computation is the connection 1-form
$\varpi$ on $\mathit{OM}$.
A connection 1-form assigns a skew symmetric
matrix to every tangent vector on $\mathit{OM}$ and it satisfies the following
conditions:

(1) $\varpi(A^*)=A$ for all $A\in\so(n)$;

(2) for all $a\in O(n)$ and $w\in \mathit{OM}$, $\varpi( {R_{a}}_*w)=Ad(a^{-1})
\varpi(w)$.
We recall that ${R_a}_*(A^*)=(Ad(a^{-1})A)^*$ for all $a\in O(n)$.
It is convenient to consider horizontal tangent vectors on $\mathit{OM}$ as
elements of the kernel of $\varpi$.
If $\{A_1, \ldots, A_N\}$ is a basis of $\so(n)$, then the horizontal
component of a vector $w$ is $w^h=w-\sum_j \langle \varpi(w),
A_j\rangle A_j^*$.

The connection 1-form $\varpi$ is basically the set of Christoffel symbols.
Let $E=\{E_1, \ldots, E_n\}$ be a local frame; we define the
Christoffel symbols relative to $E$ by $\nabla E_j=\sum_{ki} \Gamma
^k_{ij} \,dx_i \otimes E_k$.
Let $\theta^i$ be the set of dual differential 1-forms on $M$ to $\{
E_i\}$: $\theta^i (E_j)=\delta_{ij}$.
We define $ \omega_k^i=\Gamma^i_{lk} \theta^l$. Then $d\theta
^i=-\sum_k
\omega_k^i\wedge\theta^k $.
Let $\{A_i^j\}$ be a basis of ${\mathfrak g}$. To each moving frame
$E$, we
associate a 1-form, $\omega=\sum_{i,j} \omega_j^iA_i^j$, on $M$.
If $(O,x)$ is a chart of $M$ and $s\dvtx O\to \mathit{OM}$ is a local section of $\mathit{OM}$,
let us denote by $\omega_s$ the differential 1-form given above, then
$\varpi(s_*v)=\omega_s(v)$.
Conditions~(1) and (2) are equivalent to the following: if $a\dvtx U\to G$
is a smooth function,
\[
\varpi\bigl( (s\cdot a)_* v\bigr)=a^{-1}(x)\,da(v)+a^{-1}(x)
\varpi( s_*v)a(x).
\]
This corresponds to the differentiation of $s\cdot a$ and this type of
consideration will be used in the next section.

\section{Some lemmas}

\begin{lemma} Let $M$ be a geodesically complete Riemannian manifold.
Let $(u_t^\varepsilon)$ be the solution to the SDE (\ref{ou-1}) on $\mathit{OM}$.
Let $x_t^\varepsilon=\pi(u_t^\varepsilon)$, which
has a unique horizontal lift, $\tilde x_t^\varepsilon$, through
$u_0\equiv u_0^\varepsilon$. Then
\begin{eqnarray*}
{d\over dt} \tilde x_t^\varepsilon&=&H_{\tilde x_t^\varepsilon}
\bigl(g_t^\varepsilon e_0\bigr),
\\
dg_t^\varepsilon&=& {1\over\sqrt\varepsilon} \sum
_{k=1}^m{g_t^\varepsilon}
A_k \circ dw_t^k+g_t^\varepsilon
\bar A \,dt,
\end{eqnarray*}
where $g_0^\varepsilon$ is the unit matrix.
Consequently the SDE (\ref{ou-1}) is conservative.
\end{lemma}

\begin{pf} By the defining properties of the basic horizontal vector fields,
$\dot x_t^\varepsilon=\pi_*(H_{u_t^\varepsilon}(e_0))=u_t^\varepsilon e_0$.
Let $\h_u(v)$ denote the horizontal lift of a tangent vector $v$
through $u\in \mathit{OM}$.
Since $u_t^\varepsilon e_0$ has unit speed, the solution exists for all
time if $(u_t^\varepsilon)$ does, and
\[
{d\over dt} \tilde x_t^\varepsilon=
\h_{\tilde x_t^\varepsilon}\bigl(\dot x_t^\varepsilon\bigr)=
\h_{\tilde x_t^\varepsilon}\bigl(u_t^\varepsilon e_0\bigr).
\]
At each time $t$, the horizontal lift $(\tilde x_t^\varepsilon)$ of the
curve $(x_t^\varepsilon)$ through $u_0$
and the original curve $u_t^\varepsilon$ belong to the same fibre. Let
$g_t^\varepsilon$ be an element of $G$ with the property that
$u_t^\varepsilon=\tilde x_t^\varepsilon g_t^\varepsilon$. Then $g_0^\varepsilon$ is
the unit matrix and
\[
{d\over dt} \tilde x_t^\varepsilon=\h_{\tilde x_t^\varepsilon}
\bigl(\tilde x_t^\varepsilon g^\varepsilon_t
e_0\bigr)=H_{\tilde x_t^\varepsilon}\bigl(g_t^\varepsilon
e_0\bigr).
\]
If $a_t$ is a $C^1$ path with values in $O(n)$, $a_t^{-1} \dot a_t=
{d\over dr} |_{r=0}e^{r a_t^{-1}\dot a_t}$,
its action on $u$ gives rise to a fundamental vector field,
\[
{d\over dt} \bigg|_t ua_t={d\over dr}\bigg|_{r=0}
ua_t a_t^{-1}a_{r+t} =
\bigl(a_t^{-1}\dot a_t\bigr)^*(ua_t).
\]
We denote by $\mathit{DL}_g$ and $\mathit{DR}_g$, respectively, the differentials of the
left multiplication and of the right action.
By It\^o's formula applied to the product $\tilde x_t^\varepsilon
g_t^\varepsilon$,
%
\[
du_t^\varepsilon=\mathit{DR}_{g_t^\varepsilon} \circ d\tilde
x_t^\varepsilon+ \bigl(\mathit{DL}_{(g_t^\varepsilon)^{-1}}\circ dg_t^\varepsilon
\bigr)^*\bigl(u_t^\varepsilon\bigr).
\]
Since right translation of horizontal vectors are horizontal, the
connection 1-form vanishes on the first term and
$\varpi(\circ\, du_t^\varepsilon) =\mathit{DL}_{(g_t^\varepsilon)^{-1}}\circ
dg_t^\varepsilon$. We apply $\varpi$ to the SDE for $u_t^\varepsilon$,
\begin{eqnarray*}
dg_t^\varepsilon&=&\mathit{DL}_{g_t^\varepsilon} \varpi\bigl(\circ
\,du_t^\varepsilon\bigr) =\mathit{DL}_{g_t^\varepsilon}\varpi \Biggl(
{1\over\sqrt\varepsilon} \sum_{k=1}^NA_k^*
\bigl(u_t^\varepsilon\bigr)\circ dw_t^k+
\bar A^*\bigl(u_t^\varepsilon\bigr) \,dt \Biggr)
\\
&=& {1\over\sqrt\varepsilon} \sum_{k=1}^m{g_t^\varepsilon}
A_k \circ dw_t^k+g_t^\varepsilon
\bar A \,dt.
\end{eqnarray*}
There is a global solution to the above equation. The ODE $ {d\over dt}
\tilde x_t^\varepsilon=H_{\tilde x_t^\varepsilon} (g_t^\varepsilon e_0)$
has bounded right-hand side and has a global solution. It follows that
$u_t^\varepsilon=\tilde x_t^\varepsilon g_t^\varepsilon$ has a global solution.
\end{pf}

\begin{remark}
Since the stochastic process $(g_t^\varepsilon)$ is sample continuous with
initial value the unit matrix, it stays in the connected component
$\mathit{SO}(n)$ of $O(n)$.
\end{remark}

If $\{A_k\}$ is an orthonormal basis of $\so(n)$ let $\L_G ={1\over2}
\sum_{k=1}^N L_{gA_{k}} L_{gA_{k}}$.
Then $(g_t^\varepsilon)$ is a Markov process with infinitesimal generator
\[
\L^\varepsilon={ 1\over\varepsilon} \L_G+L_{g\bar A}.
\]

\begin{lemma}
\label{lem2}
Let $M$ be a complete Riemannian manifold with positive injectivity
radius. Suppose that there are numbers $C>0$ and
$a_2>0$ such that $\sup_{\rho(x,y) \le a_2} |\nabla\, d \rho|(x,y) \le C$.
Let $T>0$.
The probability distributions of the family of stochastic processes $\{
\tilde x_{t/\varepsilon}^\varepsilon, t\le T\}$ are tight.
There is a metric $\tilde d$ on $M$ such that
$ \{(\tilde x_{t/\varepsilon}^\varepsilon) \}$ is equi-H\"older
continuous with exponent $\alpha<{1\over2}$.
\end{lemma}

\begin{pf}
Let $\mu^\varepsilon$ be the probability laws of $(\tilde x_t^\varepsilon)$
on the path space over $\mathit{OM}$ with initial value $u_0$, which we denote
by $C([0,T]; \mathit{OM})$.
Since $\tilde x_0^\varepsilon=u_0$, it suffices to estimate the modulus
of continuity and show that for all positive numbers $a, \eta$, there
exists $\delta>0$ such that for all $\varepsilon$ sufficiently small (see
Billingsley \cite{Billingsley-68} and Ethier and Kurtz \cite{Ethier-Kurtz86})
\[
P\Bigl(\omega\dvtx \sup_{|s-t|<\delta} d\bigl(\tilde
x_t^\varepsilon, \tilde x_s^\varepsilon \bigr)>a
\Bigr)<\delta\eta.
\]
Here, $d$ denotes a distance function on $\mathit{OM}$. We will choose a
suitable distance function.
The Riemannian distance function $\tilde\rho(x,y)$ is not smooth in
$y$ if $y$ is in the cut locus of $x$.
To avoid any assumption on the cut locus
of $\mathit{OM}$, we construct a new distance function that preserves the
topology of $\mathit{OM}$.

Let $2a$ be the minimum of $1$, $a_2$ and the injectivity radius of $M$.
Let $\phi\dvtx  \R_+\to\R_+$ be a smooth concave function such that
$\phi
(r)=r$ when $r<a$ and $\phi(r)=1$ when $r\ge2a$.
Let $\rho$ and $\tilde\rho$ be, respectively, the Riemannian distance
on $M$ and on $\mathit{OM}$. Then $\phi\circ\rho$ and $\tilde d=\phi\circ
\tilde\rho$ are distance functions
on $M$ and on $\mathit{OM}$, respectively. Then for $r<t$,
\[
\phi^2\circ\tilde\rho\bigl(\tilde x_{t/\varepsilon}^\varepsilon,
\tilde x_{r/\varepsilon}^\varepsilon\bigr) =\int_{r/\varepsilon}^{t/\varepsilon}
D \bigl(\phi^2\circ \tilde \rho\bigl(\tilde x_r^\varepsilon,
\cdot\bigr) \bigr)_{ \tilde x_s^\varepsilon} \bigl( H_{\tilde x_s^\varepsilon} \bigl(g_s^\varepsilon
e_0\bigr) \bigr)\,ds.
\]
Since $H_{\tilde x_s^\varepsilon} (g_s^\varepsilon e_0)$ has unit length,
from the equation above we do not
observe, directly, a uniform bound in $\varepsilon$.

For further estimates, we work with a $C^2$ function $F\dvtx \mathit{OM}\to\R$ to
simplify the notation. Also,
the computations below and some of the identities will be used later in
the proof of Theorem~\ref{thm1}.
Let $0\le r<t$,
%
\begin{equation}
\label{3.3-1} F\bigl(\tilde x_{t/\varepsilon}^\varepsilon\bigr)=F\bigl(\tilde
x_{r/\varepsilon
}^\varepsilon\bigr)+\int_{r/\varepsilon}^{t/\varepsilon}(
DF)_{\tilde
x_s^\varepsilon} \bigl(H_{\tilde x_s^\varepsilon}\bigl( g_s^\varepsilon
e_0\bigr) \bigr) \,ds.
\end{equation}

Let $\{e_i\}$ be an orthonormal basis of $\R^n$. We define two sets of
functions $f_i\dvtx \mathit{OM}\to\R$ and $h_i\dvtx O(n)\to\R$:
\[
f_i(u)=( DF)_{u} (H_{u} e_i ),\qquad
\alpha_i(g)=\langle g e_0, e_i\rangle.
\]
From the linearity of $H_u$, we obtain the identity $H_u(ge_0)=\sum_{i=1}^n H_u(e_i)\alpha_i(u)$.
Thus, the integrand in (\ref{3.3-1}) factorizes and we have
%
\begin{equation}
\label{3.3-3} F\bigl(\tilde x_{t/\varepsilon}^\varepsilon\bigr)=F\bigl(\tilde
x_{r/\varepsilon
}^\varepsilon\bigr) +\sum_{i=1}^n
\int_{r/\varepsilon}^{t/\varepsilon} f_i\bigl(\tilde
x_s^\varepsilon\bigr) \alpha_i\bigl(
g_s^\varepsilon e_0\bigr) \,ds.
\end{equation}

Since the Riemannian metric on $G=\mathit{SO}(n)$ is bi-invariant, the Riemannian
volume measure,
which locally has the form $\sqrt{\operatorname{det}(g_{ij})}\,dx^1\wedge\cdots\wedge
\,dx^N$, is the Haar measure.
Let $dg$ be the Haar measure normalized to be a probability measure on $G$.
Let $\tilde g$ be a rotation such that $\tilde ge_0=-e_0$. Then
$\int_{G} g (\tilde ge_0) \,dg= \int_{G} g (e_0) \,dg$.
The integral of $ge_0$ with respect to the Haar measure vanishes.
In particular, $\int_G \alpha_i \,dg=0$. On a compact Riemannian
manifold the Poisson equation with a smooth function
that is centered with respect to the Riemannian volume measure has a
unique centered smooth solution. For each $i$, let $h_i\dvtx G\to\R$
be the smooth centred solution to the Poisson equation
%
\begin{equation}
\label{Poisson} \L_G h_i=\alpha_i= \langle
ge_0, e_i\rangle.
\end{equation}
We apply It\^o's formula to the function $f_ih_i$ and $r<t$,
\begin{eqnarray*}
f_i \bigl(\tilde x_{t/\varepsilon}^\varepsilon\bigr)h_i
\bigl( g_{t/
\varepsilon}^\varepsilon\bigr)& =&f_i\bigl(\tilde
x_{r/\varepsilon}^\varepsilon\bigr) h_i\bigl(g_{r/\varepsilon
}^\varepsilon
\bigr) + \int_{r/\varepsilon}^{t/\varepsilon} (Df_i)_{ \tilde
x_s^\varepsilon
}
\bigl(H_{\tilde x_s^\varepsilon} \bigl( g_s^\varepsilon e_0
\bigr) \bigr) h_i \bigl(g_s^\varepsilon\bigr)\,ds
\\
&&{}+{1\over\sqrt\varepsilon} \sum_k \int
_{r/\varepsilon}^{t/
\varepsilon
} f_i \bigl(\tilde
x_s^\varepsilon\bigr) (Dh_i)_{( g_s^\varepsilon)}
\bigl(g_s^\varepsilon A_k\bigr) \,dw_s^k
\\
&&{}+\int_{r/\varepsilon}^{t/\varepsilon} f_i \bigl(\tilde
x_s^\varepsilon\bigr) L_{g_s^\varepsilon\bar A } h_i\bigl(
g_s^\varepsilon\bigr)\,ds +{1\over\varepsilon} \int
_{r/\varepsilon}^{t/\varepsilon} f_i \bigl(\tilde
x_s^\varepsilon\bigr) \L_G h_i\bigl(
g_s^\varepsilon\bigr)\,ds.
\end{eqnarray*}

We sum up the above equation from $i=1$ to $n$. Note that
\[
\sum_{i=1}^n f_i(u)
\L_G h_i(g)=\sum_{i=1}^n
f_i(u)\alpha_i(g).
\]
We compare the last term in the above formula for $f_i (\tilde
x_{t/\varepsilon}^\varepsilon)h_i( g_{t/\varepsilon}^\varepsilon)$ with
the integral in (\ref{3.3-3}) to obtain that
\begin{eqnarray*}
F\bigl(\tilde x_{t/\varepsilon}^\varepsilon\bigr)&=&F\bigl(\tilde
x_{r/\varepsilon
}^\varepsilon\bigr)+ \varepsilon\sum
_{i=1}^n \bigl(f_i \bigl(\tilde
x_{t/
\varepsilon
}^\varepsilon\bigr)h_i\bigl( g_{t/\varepsilon}^\varepsilon
\bigr)-f_i\bigl(\tilde x_{r/
\varepsilon}^\varepsilon\bigr)
h_i\bigl(g_{r/\varepsilon}^\varepsilon\bigr) \bigr)
\\
&&{}-\varepsilon\sum_{i=1}^n\int
_{r/\varepsilon}^{t/\varepsilon} (Df_i)_{ \tilde x_s^\varepsilon}
\bigl(H_{\tilde x_s^\varepsilon} \bigl( g_s^\varepsilon e_0
\bigr) \bigr) h_i \bigl(g_s^\varepsilon\bigr)\,ds
\\
&&{}-\varepsilon\sum_{i=1}^n\int
_{r/\varepsilon}^{t/\varepsilon} f_i \bigl(\tilde
x_s^\varepsilon\bigr) L_{g_s^\varepsilon\bar A } h_i\bigl(
g_s^\varepsilon\bigr)\,ds
\\
&&{}-\sqrt\varepsilon\sum_{i=1}^n\sum
_{k=1}^N \int_{r/\varepsilon
}^{t/\varepsilon}
f_i \bigl(\tilde x_s^\varepsilon\bigr)
(Dh_i)_{( g_s^\varepsilon
)}\bigl(g_s^\varepsilon
A_k\bigr) \,dw_s^k.
\end{eqnarray*}

Let us compute the differential of $f_i(u)=( DF)_{u}  (H_{u}
e_i )$. Let $\nabla$ be the flat connection on $\mathit{OM}$. It is
determined by the parallelization
$\mathbb X\dvtx \mathit{OM}\times\R^n\times\so(n)\to T\mathit{OM}$ where ${\mathbb X}_u(e,
A) = H_u(e)+\varpi_u^{-1}(A)$. In the calculation below, we use the
fact that $\nabla H(e)=0$.
%
\begin{eqnarray}
\label{generator-1} &&F\bigl(\tilde x_{t/\varepsilon}^\varepsilon\bigr)
-F\bigl(\tilde
x_{r/\varepsilon
}^\varepsilon\bigr)
\nonumber
\\
&&\qquad= \varepsilon\sum_{i=1}^n \bigl( (
DF)_{\tilde x_{t/\varepsilon
}^\varepsilon} (H_{\tilde x_{t/\varepsilon}^\varepsilon} e_i ) h_i\bigl(
g_{t/\varepsilon}^\varepsilon\bigr) -( DF)_{\tilde x_{r/\varepsilon}^\varepsilon}
 (H_{\tilde
x_{r/
\varepsilon}^\varepsilon}
e_i ) h_i\bigl( g_{r/\varepsilon}^\varepsilon \bigr)
\bigr)
\nonumber
\\
&&\quad\qquad{}-\varepsilon\sum_{i=1}^n\int
_{r/\varepsilon}^{t/\varepsilon} (\nabla DF)_{ \tilde x_s^\varepsilon}
\bigl(H_{\tilde x_s^\varepsilon}\bigl( g_s^\varepsilon e_0
\bigr), H_{\tilde x_s^\varepsilon} ( e_i) \bigr) h_i
\bigl(g_s^\varepsilon\bigr)\,ds
\\
&&\qquad\quad{}-\varepsilon\sum_{i=1}^n\int
_{r/\varepsilon}^{t/\varepsilon} ( DF)_{\tilde x_s^\varepsilon} (H_{\tilde x_s^\varepsilon}
e_i ) L_{g_s^\varepsilon\bar A } h_i\bigl( g_s^\varepsilon
\bigr)\,ds
\nonumber
\\
&&\qquad\quad{}-\sqrt\varepsilon\sum_{i=1}^n\sum
_{k=1}^N \int_{r/\varepsilon
}^{t/\varepsilon}
( DF)_{\tilde x_{s}^\varepsilon} (H_{\tilde
x_{s}^\varepsilon} e_i ) (Dh_i)_{( g_s^\varepsilon)}
\bigl(g_s^\varepsilon A_k\bigr) \,dw_s^k.\nonumber
\end{eqnarray}
We also remark that $|H_{\tilde x_s^\varepsilon} e_i|=1 $, $|H_{\tilde
x_s^\varepsilon} g_s^\varepsilon e_i|=1$, $|g_s^\varepsilon\bar A|=|\bar A|$.
If $F$ is a function that is $BC^2$, by the Kunita--Watanabe
inequality, for any $p\ge1$,
\[
\E\bigl\llvert F\bigl(\tilde x_{t/\varepsilon}^\varepsilon\bigr)-F\bigl(\tilde
x_{r/
\varepsilon}^\varepsilon\bigr)\bigr\rrvert ^p \le
C_1(T)\varepsilon^p \bigl(|DF |_\infty+ |\nabla
DF|_\infty \bigr)+ C_1(T) |DF|_\infty
|t-r|^{p/2},
\]
for some constant $ C_1(T)$. If $\varepsilon^2 \le|t-r|$, there exists a
constant $C_2(T)$, such that
$\E\llvert F(\tilde x_{t/\varepsilon}^\varepsilon)-F(\tilde x_{r/
\varepsilon}^\varepsilon)\rrvert ^p \le C_2(T)|t-r|^{p/2}$.
If $|t-r|<\varepsilon^2$, we estimate directly from (\ref{3.3-1}):
\[
\bigl|F\bigl(\tilde x_{t/\varepsilon}^\varepsilon\bigr)-F\bigl(\tilde
x_{r/\varepsilon
}^\varepsilon\bigr)\bigr|\le C {t-r \over\varepsilon}\le C
\sqrt{t-r}.
\]
Thus, for $C(T)=C_2(T)+C^p$,
\[
\E\bigl\llvert F\bigl(\tilde x_{t/\varepsilon}^\varepsilon\bigr)-F\bigl(\tilde
x_{r/
\varepsilon}^\varepsilon\bigr)\bigr\rrvert ^p \le
C(T)|t-r|^{p/2}.
\]
We apply the above formula to $F=\phi^2 \circ\tilde\rho(\cdot, u_0)$
where $u_0=\tilde x_0^\varepsilon$.
Since $\phi$ is bounded so is $F$. Since $|\nabla\tilde\rho(\cdot,
u_0)|\le1$ and $\phi'$ is bounded,
$\nabla F=2\phi\phi' \nabla\rho(\cdot, u_0)$ is bounded. The norm of
its second derivative is
\[
\bigl| 2 \bigl(\phi'\bigr)^2 \nabla\rho\otimes\nabla\rho+2
\bigl( \phi\phi''\bigr) \nabla \rho \otimes\nabla
\rho+2\bigl(\phi\phi'\bigr)\nabla\, d\rho\bigr|,
\]
and the tensor is evaluated at $\rho(x,y)$. We remark that $\phi
'(x,y)=0$ when $\rho(x,y) \ge a$ and $|\nabla\, d\rho(\rho(x,y))|\le
C$ when
$\rho(x,y) \ge a$. Hence, for all $u_0$, there is a common number
$C(T)$ s.t.
\[
\E\bigl\llvert \tilde d\bigl(\tilde x_{t/\varepsilon}^\varepsilon,
u_0\bigr)\bigr\rrvert ^p \le C(T)t^{p/2}.
\]
Conditioning on $\F_r$ to see that
\[
\E\bigl\llvert \tilde d\bigl(\tilde x_{t/\varepsilon}^\varepsilon, \tilde
x_{r/
\varepsilon}^\varepsilon\bigr)\bigr\rrvert ^p \le
C(T)|t-r|^{p/2}.
\]
The tightness of the law of $\{\tilde x_{t/\varepsilon}^\varepsilon\}$
follows. By Kolmogorov's
criterion, $\{\tilde x_{t/\varepsilon}^\varepsilon\}$ is H\"older
continuous with exponent $\alpha$ for any
$\alpha<{1\over2}$. The H\"older constants are independent of
$\varepsilon$ and, for any $p'<p$,
Kolmogorov's criterion yields
%
\begin{equation}
\sup_\varepsilon\E\sup_{s\neq t} \biggl(
{\tilde d(\tilde x_{t/
\varepsilon}^\varepsilon, \tilde x_{s/\varepsilon}^\varepsilon) \over
|t-s|^\alpha} \biggr)^{p'}<\infty,
\end{equation}
thus completing the proof.
\end{pf}

We will need the following lemma in which we make a statement on the
limit of a function
of two variables, one of which is ergodic and the other one varies
significantly slower.
The result is straightforward, but we include the proof for completeness.
If $f\dvtx N\to\R$ is a Lipschitz continuous function on a metric space
$(N, d)$ with distance function $d$, we denote
by $|f|_{\Lip}$ its Lipschitz semi-norm. If $S$ is a subset of $N$, we
let $\Osc_S(f)$ denote
$|\sup_{x\in S} f(x)-\inf_{x\in S} f(x)|$, the Oscillation of $f$ over
$S$. Let $\Osc(f)=\Osc_N(f)$.

Let $E(N)$ be one of the following classes of real valued functions on
a metric space $(N,d)$:
\[
E(N)=\bigl\{ f\dvtx N\to\R\dvtx |f|_{\Lip}<\infty, \Osc(f) <\infty\bigr
\}
\]
or $E_r(N)=E(N)\cap C^r$, where $r=0,1, \ldots,\infty$. Denote
\[
|f|_{E}=|f|_{\Lip}+\Osc(f).
\]
Let $d$ be the metric with respect to which the Lipschitz property is
defined. We define $\tilde d=d\wedge1$
to be a new metric on $N$. Then $|f|_{\Lip} \le C$ and $\Osc(f)\le C$
is equivalent to $f$ being Lipschitz with respect
to $\tilde d$.

Let $p\ge1$ and let $W_p(N)$ denote the Wasserstein $p$-distance
between two probability measures on a metric space
$(N, d)$:
\[
\bigl(W_p(\mu_1, \mu_2)
\bigr)^{p}=\inf_{\{\nu\dvtx  (\pi_1)_*\nu
=\mu_1, (\pi
_2)_*\nu=\mu_2\}}\int_{N\times N}
\bigl(d(x,y)\bigr) ^p \,d\nu(x,y).
\]
Let $\mu^\varepsilon, \mu$ be a family of probability measures on the
metric space $(N,d)$. 
Then $\mu^\varepsilon\to\mu$ in $W_p(N)$ if and only if they converge
weakly and\break  $\sup_{x\in N} \int(d(x, y))^p\,d\mu_\varepsilon(y)$
is bounded for any $x\in N$. If $\tilde d=d\wedge1$, then $\tilde d$
and $d$ induce the same topology on $N$ and
the concepts of weak convergence are equivalent. With respect to
$\tilde d$, weak convergence is equivalent to
Wasserstein $p$-convergence.

Let $(\Omega, \F, (\F_t),P)$ be a filtered probability space. Let
$(Y,\rho), (Z, d)$ be metric spaces or $C^m$ manifolds.
Let $\{(y_t^\varepsilon, t\le T), \varepsilon>0\}$ be a family of $\F
_t$-adapted stochastic processes with state space $Y$.
Let $ (z_t^\varepsilon)$ be a family of sample continuous $\F_t$-Markov
processes on $Z$.

\begin{assumption}\label{assum}
(1) The stochastic processes $(y_{t/\varepsilon} ^\varepsilon, t\le T)$ are
equi-uniformly continuous and converge weakly to a continuous process
$(\bar y_t, t\le T)$.\vspace*{-6pt}
\begin{longlist}[(2)]
\item[(2)]
For each $\varepsilon$, $(z_{t\varepsilon}^\varepsilon, t\le T)$ has an
invariant measure $\mu_\varepsilon$. There exists a function $\delta$ on
$\R_+\times Z\times\R_+$ with the property that $\delta(\cdot, z,
\varepsilon)$ is nondecreasing for each pair of $(z, \varepsilon)$ and
$\lim_{\varepsilon\to0}\sup_{z\in Z} \delta(K, z, \varepsilon)=0$ for all $K$
and for all $f\in E_r(Z)$ and $t>0$,
\[
\E\biggl\llvert {\varepsilon\over t }\int_{0}^{t/\varepsilon}
f\bigl(z_{s\varepsilon
}^\varepsilon\bigr)\,ds-\int_Z f
(z) \,d\mu_\varepsilon(z) \biggr\rrvert \le\delta\biggl(|f|_E,
z_0^\varepsilon,{\varepsilon\over t}\biggr).
\]
\item[(3)] There exists a probability measure $ \mu$ on $W^1(
C([0,T];Z)) $ s.t.\break $\lim_{\varepsilon\to0} W_1( \mu_\varepsilon, \mu)=0$.
\item[(4)] The processes $(y^\varepsilon_{t/\varepsilon})$ converges to
$(\bar y_t)$
in $W_1(Y)$, and there exists an exponent $\alpha>0$ such that
\[
\sup_\varepsilon\E \biggl(\sup_{s\neq t}
{\rho(y_{t/\varepsilon
}^\varepsilon, y_{s/\varepsilon}^\varepsilon) \over|t-s|^\alpha} \biggr)<\infty.
\]
\end{longlist}
\end{assumption}

We cannot assume that $(\bar y_t)$ is adapted to the filtration
with respect to which $(z^\varepsilon_{t/\varepsilon})$ is a Markov
process. The process
$(z_{t/\varepsilon}^\varepsilon)$ is usually not convergent and we do
not assume that $( y_t^\varepsilon, z_t^\varepsilon)$ and $(\bar y_t)$ are
realized in the same
probability space.

We denote by $\hat P_\eta$ the probability distribution of a random
variable $\eta$ and let $T$ be a positive real number.
If $r$ is a positive number, let $C([0,r];Y)$ denote the space of
continuous paths, $\sigma\dvtx  [0,r]\to Y$, on $Y$. If $F\dvtx C([0,r];Y)\to
\R$
is a Borel measurable function,
we use the shorter notation $F(y^\varepsilon_ {\cdot/\varepsilon})$ for
$F ( (y^\varepsilon_{u/\varepsilon}, u\le r)  )$.

\begin{lemma}\label{lem3}
Let $(\Omega, \F, (\F_t),P)$ be a filtered probability space. Let
$(Y,\rho),\break  (Z, d)$ be metric spaces or $C^m$ manifolds in case $m\ge1$.
Let $\{(y_t^\varepsilon, t\le T), \varepsilon>0\}$ be a family of $\F
_t$-adapted stochastic processes on $Y$.
Let $ (z_t^\varepsilon)$ be a family of sample continuous $\F_t$-Markov
processes on $Z$.
Let $G\in E_m(Y\times Z)$. Let $0\le r<t$ and let $F\dvtx C([0,r];Y)\to\R$
be a bounded continuous function.
We define
\[
A(\varepsilon)\equiv A(\varepsilon,F,G):=F\bigl(y^\varepsilon_{\cdot/
\varepsilon}
\bigr) \int_r^t G\bigl(y^\varepsilon_{s/\varepsilon},
z^\varepsilon_{s/
\varepsilon}\bigr) \,ds.
\]
\begin{itemize}
\item If (1)--(3) in Assumption~\ref{assum} hold, then the random
variables $A(\varepsilon)$ converge weakly to $A$ as $\varepsilon\to0$, where
\[
A\equiv A(F,G):= F( \bar y_\cdot) \int_r^t
\int_Z G(\bar y_s, z) \,d\mu(z)\,ds.
\]
\item Assume (1)--(4) in Assumption~\ref{assum}. Then there is a
constant $c$, s.t. for $\varepsilon<1$,
\begin{eqnarray*}
&&W_1 (\hat P_{A(\varepsilon)}, \hat P_A )\\
&&\qquad\le c
|F|_\infty \max_{z\in Z} \delta \biggl(
|G|_E, z, {\varepsilon\over t-r} \biggr) +2 \varepsilon|F|_\infty
\min\bigl( |G|_\infty, \bigl|\Osc(G)\bigr|\bigr)
\\
&&\qquad\quad{}+c(t-r)|F|_\infty|G|_{\Lip} \bigl(W_1 \bigl(\hat
P_{y^\varepsilon
_{\cdot/\varepsilon}}, \hat P_{\bar y_\cdot} \bigr)+ W_1\bigl(\mu
^\varepsilon, \mu\bigr) \bigr)+c\varepsilon^\alpha|F|_\infty|G|_{\Lip}.
\end{eqnarray*}
\end{itemize}
\end{lemma}

\begin{pf}
Let us fix the functions $F$, $G$, $r$, $t$ and define
\begin{eqnarray*}
{\mathcal E_1}(r,t)&=& \int_r^t G
\bigl(y_{s/\varepsilon}^\varepsilon, z^\varepsilon_{s/\varepsilon} \bigr)
\,ds-\int_r^t \int_Z G
\bigl(y_{s/
\varepsilon
}^\varepsilon, z\bigr) \,d\mu_\varepsilon(z)\,ds;
\\
{\mathcal E_2}&=&F\bigl(y^\varepsilon_{\cdot/\varepsilon} \bigr)
\biggl( \int_r^t \int_Z
G\bigl(y_{s/\varepsilon}^\varepsilon, z\bigr) \,d\mu_\varepsilon(z)\,ds-\int
_r^t \int_Z G
\bigl(y_{s/\varepsilon}^\varepsilon, z\bigr) \,d\mu(z)\,ds \biggr);
\\
I(\varepsilon)&=&F\bigl(y^\varepsilon_{\cdot/\varepsilon} \bigr)\int
_r^t \int_Z G
\bigl(y_{s/\varepsilon}^\varepsilon, z \bigr)\,d\mu(z) \,ds.
\end{eqnarray*}
The proof is split into three parts: (i) $F(y^\varepsilon_{\cdot/
\varepsilon} ) {\mathcal E_1}(r,t)$ converges to zero in $L_p(\Omega)$ for
any $p>1$,
(ii) $\error_2$ converges to zero in $L_p(\Omega)$ for any $p>1$ and
(iii) $I(\varepsilon)$ converges to $A$ weakly.

We first prove that $F (y^\varepsilon([0, {r\over\varepsilon}])
)\error_1(r,t)$
converges to zero in $L_p(\Omega)$.
Since $F$ is bounded it is sufficient to take $r=0$ and $F$ a constant,
and to work with $\error_1(0,t)$. Let us write
\[
{\mathcal E_1}:= \int_0^t G
\bigl(y_{s/\varepsilon}^\varepsilon, z^\varepsilon _{s/\varepsilon} \bigr)
\,ds -\int_0^t \int_Z G
\bigl(y_{s/\varepsilon}^\varepsilon, z\bigr) \,d\mu_\varepsilon(z)\,ds.
\]

Let $0=t_0<t_1<\cdots<t_M\le t$ be a partition of $[0, {t}]$ into
pieces of size ${t \varepsilon}$. Let $M\equiv M_\varepsilon=[{1\over
\varepsilon}]$.
Let $\Delta t_i=t_{i+1}-t_i$ and let $\tilde t={ t\varepsilon}M_\varepsilon
$. Below $a\sim b$ indicates ``$a-b=O(\varepsilon)$''
as $\varepsilon$ converges to $0$. Since $G\in E_m(Y\times Z)$,
\begin{eqnarray*}
\bigl\llvert \error_1(\tilde t,t)\bigr\rrvert &\le&2 \min \biggl(
|G|_\infty, \bigl|\Osc (G)\bigr|, |G|_{\Lip}\max_{0\le s \le t}
\int_Z\,d\bigl(z^\varepsilon_{s/
\varepsilon
},z\bigr)
\mu_\varepsilon(dz) \biggr) (t-\tilde t)
\\
&\le&\varepsilon2\min\bigl( |G|_\infty, \bigl|\Osc(G)\bigr|\bigr) \le2
\varepsilon\bigl(|G|_E\bigr).
\end{eqnarray*}

By the Lipschitz continuity of $G$, for each $\varepsilon>0$ the
following holds:
\begin{eqnarray*}
\error_3&:=&\Biggl\llvert \sum_{i=0}^{M_\varepsilon-1}
\int_{t_i}^{t_{i+1}} G\bigl( y^\varepsilon_{s/\varepsilon}
,z^\varepsilon_{s \varepsilon}\bigr)\,ds - \sum_{i=0}^{M_\varepsilon-1}
\int_{t_i}^{t_{i+1}} G\bigl(y^\varepsilon
_{t_i/\varepsilon}, z^\varepsilon_{s \varepsilon} \bigr)\,ds \Biggr\rrvert
\\
&\le&|G|_{\Lip} \sum_{i=0}^{M_\varepsilon-1}
\int_{t_i}^{t_{i+1}} \rho\bigl( y^\varepsilon_{s/\varepsilon},
y^\varepsilon_{t_i/\varepsilon}\bigr) \,ds.
\end{eqnarray*}
By equi-uniform continuity of $(y^\varepsilon_{s/\varepsilon})$, for
almost surely all $\omega$, $\error_3$ converges to zero. Since
$\error_3$
is bounded the convergence is in $L_p(\Omega)$.
If $(y^\varepsilon_{s/\varepsilon})$ is assumed to be equi-H\"older
continuous as in condition (4), there is a convergence rate of
$\varepsilon
^\alpha|G|_{\Lip}$
for the $L^p$ convergence.

We prove next that $ \sum_{i=0}^{M_\varepsilon-1} \int_{t_i}^{t_{i+1}} G(
y^\varepsilon_{t_i/\varepsilon}, z^\varepsilon_{s/\varepsilon} ) \,ds $
converges.
We apply the Markov property of $(z_t^\varepsilon)$ and we use the fact
that $(y_t^\varepsilon)$ is adapted to the filtration $(\F_t)$, with
respect to which
$(z_t^\varepsilon)$ is a Markov process:
\begin{eqnarray*}
&& \sum_{i=1}^{M_\varepsilon-1} \E\biggl\llvert \int
_{t_i}^{t_{i+1}} G\bigl( y^\varepsilon_{t_i /\varepsilon},
z^\varepsilon_{s/\varepsilon} \bigr) \,ds - \Delta t_i \int
_Z G\bigl(y^\varepsilon_{t_i /\varepsilon}, z\bigr)\,d\mu
_\varepsilon (z)\biggr\rrvert
\\
&&\qquad \le\sum_{i=1}^{M_\varepsilon-1} \Delta t_i
\E \biggl( \E \biggl\{ \biggl\llvert {1\over\Delta t_i} \int
_{t_i}^{t_{i+1}} G\bigl(y^\varepsilon_{t_i/
\varepsilon},
z^\varepsilon_{s/\varepsilon} \bigr) \,ds\\
&&\hspace*{112pt}{} -\int_Z G
\bigl(y^\varepsilon_{t_i /\varepsilon}, z\bigr)\,d\mu_\varepsilon(z) \biggr
\rrvert \Big| \F_{t_i/\varepsilon} \biggr\} \biggr)
\\
&&\qquad= \sum_{i=1}^{M_\varepsilon-1} \Delta t_i
\E \biggl( \E \biggl( \biggl\llvert {\varepsilon^2\over\Delta t_i} \int
_{t_i/
\varepsilon
^2}^{t_{i+1}/\varepsilon^2} G\bigl(y, z^\varepsilon_{s \varepsilon}
\bigr) \,ds \\
&&\hspace*{121pt}{}-\int_Z G(y, z) \,d\mu_\varepsilon(z) \biggr
\rrvert \biggr) \Big|_{y=
y^\varepsilon
_{t_i/\varepsilon}} \biggr).
\end{eqnarray*}
Since ${\varepsilon^2\over\Delta t_i}={\varepsilon\over t}$, we may now
apply condition (2) and obtain
\begin{eqnarray*}
&& \E \biggl( \biggl\llvert {\varepsilon^2\over\Delta t_i} \int_{t_i/
\varepsilon
^2}^{t_{i+1}/\varepsilon^2}
G\bigl(y, z^\varepsilon_{s \varepsilon} \bigr) \,ds -\int_Z
G(y, z)\,d\mu_\varepsilon(z) \biggr\rrvert \biggr)
\\
&&\qquad\le\delta \biggl( \bigl\llvert G\bigl(y^\varepsilon_{t_i/\varepsilon}, \cdot
\bigr)\bigr\rrvert _E, z_{t_i /\varepsilon}^\varepsilon,
{\varepsilon\over t} \biggr) \le\delta \biggl( |G|_E,
z_{t_i /\varepsilon}^\varepsilon, {\varepsilon
\over t} \biggr).
\end{eqnarray*}
We record that
%
\begin{eqnarray}
\label{3.3-5} \error_4&:=& \E\Biggl\llvert \sum
_{i=0}^{M_\varepsilon-1} \int_{t_i}^{t_{i+1}}
G\bigl( y^\varepsilon_{t_i /\varepsilon}, z^\varepsilon_{s/\varepsilon} \bigr)
\,ds -\sum_{i=0}^{M_\varepsilon-1} \Delta t_i
\int_Z G\bigl(y^\varepsilon_{t_i
/
\varepsilon}, z\bigr)\,d
\mu_\varepsilon(z)\Biggr\rrvert
\nonumber
\\[-8pt]
\\[-8pt]
\nonumber
& \le&\max_{z\in Z} \delta \biggl( |G|_E, z,
{\varepsilon\over t} \biggr).
\end{eqnarray}
Let us define
\[
\error_5:=\sum_{i=0}^{M_\varepsilon-1}
\Delta t_i \int_Z G\bigl(y^\varepsilon_{t_i
/\varepsilon},
z\bigr)\,d\mu_\varepsilon(z) - \int_0^t \int
_Z G\bigl(y^\varepsilon_{s/\varepsilon}, z\bigr) \,d
\mu_\varepsilon(z)\,ds.
\]
By the definition of Riemann integral
\[
\error_5 \le|G|_{\Lip} \sum_{i=0}^{M_\varepsilon-1}
\Delta t_i {\Osc}_{ [s_i, s_{i+1}]}\bigl(y^\varepsilon_{s/\varepsilon}
\bigr),
\]
where $\Osc_{[a,b]}(f)$ denotes the oscillation of a function $f$ in
the indicated interval.
Since $(y^\varepsilon_{s/\varepsilon})$ is equi-uniform continuous on
$[0,T]$, $\error_5\to0$ in $L_p$.
Given H\"older continuity of $(y^\varepsilon_{s/\varepsilon})$
from condition (4), we have the quantitative estimates: $|\error
_5|_{L_p(\Omega)} \le C |G|_{\Lip} \varepsilon^\alpha$.
To summarize,
\[
\bigl|\error_1(0,t)\bigr|\le\bigl|\error_1(\tilde t,t)\bigr|+
\error_3+\error_4+\error_5.
\]
It follows that $F (y^\varepsilon_{r/\varepsilon})\error
_1(r,t)$ converges to zero.

When condition (4) holds, there is a constant $C$ such that
%
\begin{eqnarray}
&&\bigl\llvert F \bigl(y^\varepsilon_ {\cdot/\varepsilon}
\bigr)\error _1(r,t)\bigr\rrvert _{L_p(\Omega)}\nonumber\\
&&\qquad\le|F|_\infty \bigl(2\varepsilon\min\bigl( |G|_\infty, \bigl|\Osc(G)\bigr|
\bigr)+\error _3+\error_4+\error_5 \bigr)\nonumber
\nonumber
\\[-8pt]
\\[-8pt]
\nonumber
&&\qquad\le C|F|_\infty\bigl( \varepsilon^\alpha+\varepsilon\bigr)
|G|_{\Lip} +2\varepsilon |F|_\infty\min\bigl(
|G|_\infty, \bigl|\Osc(G)\bigr|\bigr)
\\
&&\qquad\quad{}+C|F|_\infty\max_{z\in Z} \delta \biggl(
|G|_E, z, {\varepsilon\over
t-r} \biggr). \nonumber
\end{eqnarray}
For any two random variables on the same probability space and with the
same state space, the
$L_p$ norm of their difference dominates their Wasserstein $p$-distance.
The random variable
\[
F \bigl(y^\varepsilon_ {r/\varepsilon} \bigr)\int_r^t
G\bigl(y_{s/
\varepsilon
}^\varepsilon, z^\varepsilon_{s/\varepsilon} \bigr)
\,ds -F \bigl(y^\varepsilon _{r/\varepsilon}\bigr) \int_r^t
\int_Z G\bigl(y_{s/\varepsilon
}^\varepsilon , z\bigr) \,d
\mu_\varepsilon(z)\,ds \stackrel{W_p(N)} { \to} 0,
\]
with the same rate as indicated above.

We proceed to step (ii). It is clear that for almost all $\omega$,
$F(y^\varepsilon_{\cdot/\varepsilon} )\int_r^t G(y^\varepsilon_{s/
\varepsilon} , z)\,ds$ is Lipschitz continuous in $z$. For any $z_1, z_2\in Z$,
\begin{eqnarray*}
& &\biggl\llvert F\bigl(y^\varepsilon_{\cdot/\varepsilon} \bigr) \int
_r^tG\bigl(y_{s/
\varepsilon}^\varepsilon,
z_1\bigr) \,ds -F\bigl(y^\varepsilon_{\cdot/\varepsilon} \bigr) \int
_r^t G\bigl(y_{s/\varepsilon
}^\varepsilon,
z_2\bigr)\,ds \biggr\rrvert
\\
& &\qquad\le|F|_\infty d(z_1,z_2) \int
_r^t \bigl|G\bigl(y_{s/\varepsilon}^\varepsilon,
\cdot\bigr)\bigr|_{\Lip} \,ds \le(t-r) d(z_1,z_2)|F|_\infty|G|_{\Lip}.
\end{eqnarray*}
By the Kantorovich duality formula, for the distance between two
probability measures $\mu_1$ and $\mu_2$,
\[
W_1(\mu_1, \mu_2)=\sup \biggl\{ \int U \,d
\mu_1-\int U \,d\mu_2 \dvtx |U|_{\Lip
}\le1 \biggr
\},
\]
we have
\[
|\error_2|\le(t-r)\cdot|F|_\infty\cdot|G|_{\Lip}
\cdot W_1\bigl(\mu ^\varepsilon, \mu\bigr).
\]

For part (iii), let $U$ be a continuous function on $C([0,T];Y)$. If
$\sigma\in C([0,T]; Y)$, let us denote
by $\sigma([0,r])$ the restriction of the path to $[0,r]$. Since $F$ is
bounded continuous and $G$ is Lipschitz continuous,
\[
\sigma\mapsto U \biggl( F\bigl(\sigma\bigl([0,r]\bigr) \bigr) \biggl( \int
_r^t \int_Z G(\sigma
_s, z )\,d\mu(z) \,ds \biggr) \biggr)
\]
is a
continuous function on $C([0,T];Y)$.
By the weak convergence of $(y^\varepsilon_{\cdot/\varepsilon})$, $\E
 ( U(I(\varepsilon)) )$ converges to $\E
(U(A(F,G)) )$ and
the random variables $I(\varepsilon)$ converge weakly to $A(F,G)$.
By now, we have proved that $ A(\varepsilon,F,G)$ converges to $A(F,G)$
weakly; we thus conclude the first part of the lemma.


Let us assume condition (4) from Assumption~\ref{assum}. In particular,
$(y^\varepsilon_{\cdot/\varepsilon})$ converges in $W_1(C([0,T];Y))$.
Let $U$ be a Lipschitz continuous function on $C([0,T];Y)$.
We define $\tilde U\dvtx C([0,T];Y)\to\R$ by
\[
\tilde U (\sigma)=U \biggl( F\bigl(\sigma\bigl([0,r]\bigr) \bigr) \biggl( \int
_r^t \int_Z G(
\sigma_s, z )\,d\mu(z) \,ds \biggr) \biggr).
\]
Let $\sigma^1, \sigma^2$ are two paths on $Y$,
\begin{eqnarray*}
&&\bigl\llvert \tilde U(\sigma_1)-\tilde U(\sigma_2)\bigr
\rrvert
\\
&&\qquad\le|U|_{\Lip}\cdot|F|_\infty\biggl\llvert \int
_r^t \int_Z G\bigl(
\sigma ^1_s, z \bigr)\,d\mu(z) \,ds - \int
_r^t \int_Z G\bigl(
\sigma^2_s, z \bigr)\,d\mu(z) \,ds\biggr\rrvert
\\
&&\qquad\le(t-r) |U|_{\Lip}\cdot|F|_\infty\cdot|G|_{\Lip}
\cdot\sup_{0\le s\le T} \rho\bigl(\sigma^1_s,
\sigma^2_s\bigr).
\end{eqnarray*}
By the Kantorovitch duality and assumption (4),
\[
W_1 (\hat P_{I(\varepsilon)}, \hat P_I )\le(t-r) \cdot
|F|_\infty\cdot|G|_{\Lip} \cdot W_1 (\hat
P_{y^\varepsilon
_{\cdot
/\varepsilon}} , \hat P_{\bar y_\cdot} ).
\]

We collect all the estimations together. Under assumptions (1)--(4), the
following estimates hold:
\begin{eqnarray*}
W_1 (\hat P_{A(\varepsilon)}, \hat P_A ) &\le&
C|F|_\infty|G|_{\Lip}\bigl(\varepsilon^\alpha+
\varepsilon\bigr) +C|F|_\infty\max_{z\in Z} \delta
\biggl( |G|_E, z, {\varepsilon\over t-r} \biggr)
\\
&&{}+C (t-r) \cdot|F|_\infty\cdot|G|_{\Lip} \cdot \bigl(
W_1 (\hat P_{y^\varepsilon_{\cdot/\varepsilon}} , \hat P_{\bar y_\cdot} )+
W_1(\mu_\varepsilon, \mu) \bigr)
\\
&&{}+2\varepsilon|F|_\infty\min\bigl( |G|_\infty, \bigl|\Osc(G)\bigr|
\bigr).
\end{eqnarray*}
We may now limit ourselves to $\varepsilon\le1$ and conclude part 2 of
the lemma.
\end{pf}

\begin{remark}
In the lemma above, we should really think that the $z^\varepsilon$
process and process $y^\varepsilon$ follow different clocks,
the former is run at the fast time scale ${1\over\varepsilon}$ and the
latter at scale $1$.
\end{remark}

\begin{example}\label{ex1}
Let $(g_s)$ be a Brownian motion on $G=\mathit{SO}(n)$, solving
\[
dg_t=\sum_{k=1}^N
L_{g_tA_k}\,dw_t^k.
\]
Here, $\{A_1, \ldots, A_N\}$ is an orthonormal basis of ${\mathfrak g}$.
In Lemma~\ref{lem3} we take $z_t^\varepsilon=g_{t/\varepsilon}$, then
condition (2) holds.
If $f$ is a Lipschitz continuous function, it is well known that the
law of large numbers holds for $\int_0^t f(g_s)\,ds$,
so does a central limit theorem. The remainder term in the central
limit theorem
is of order $\sqrt t$ and depends on $f$ only through the Lipschitz
constant $|f|_{\Lip}$.

It is easy to see that the remainder term in the law of large numbers
depends only on the Lipschitz constant of the function.
Without loss of generality, we assume that $\int f\, dg=0$.
Let $\alpha$ solve the Poisson equation: $\Delta^G \alpha=f$.
Then
\[
{1\over t}\int_0^t
f(g_s)\,ds= {1\over t}\alpha(g_t)-
{1\over t} \alpha (g_0)-\sum
_k {1\over t}\int_0^t
(D \alpha) (g_s A_k)\,d w_s^k.
\]
Since $\alpha$ is bounded, we are only concerned with the martingale term.
By Burkholder--Davis--Gundy inequality, its $L^2$ norm is bounded by
\[
{2\over t} \Biggl(\sum_{k=1}^N
\int_0^t \E \bigl( (D \alpha)
(g_s A_k) \bigr)^2 \,ds \Biggr)^{1/2}
\le{2\over t} \biggl(\int_0^t
\E|D\alpha|^2_{g_s}\,ds \biggr)^{1/2}.
\]
By elliptic estimates, $|D\alpha|$ is bounded by $|f|_{L_\infty}$.
Since $f$ is centered, it is bounded by $\Osc(f)$.
In summary,
\[
\E \biggl({1\over t}\int_0^t
f(g_s)\,ds-\int_N f(g)\,dg \biggr)^2
\le C\bigl (\Osc (f)t^{-{1/2}}\bigr)^{2}.
\]
\end{example}

In Theorem~\ref{thm1}, we may wish to add an extra drift of the form
${1\over\varepsilon}A^*$ where $A\in{\mathfrak g}$, so that $\L_G$ is
${1\over
2}\Delta^G +L_{gA}$.
Translations by orthogonal matrices are isometries, so for any $A\in
{\mathfrak g}
$ the vector field $gA$ is a killing field, and
the Haar measure remains an invariant measure for the diffusion with
infinitesimal generator ${1\over2} \Delta^G+L_{gA}$.
However, on a compact Lie group no left invariant vector field is the
gradient of a function and
${1\over2} \Delta^G+L_g A$ is no longer a symmetric operator.
In this case, we do not know how to obtain the estimate in the example.

\section{Proof}
We are ready to prove the main theorem. In Lemma~\ref{lem2}, we used a
fundamental technique to split the integral
\[
\int_{r/\varepsilon}^{t/\varepsilon} (DF)_{\tilde x_s^\varepsilon}
(H_{x_s^\varepsilon}) \bigl(g_s^\varepsilon e_0
\bigr)\,ds
\]
into the sum of a process of finite variation and a martingale. The
computation in the proof of
Lemma~\ref{lem2} will be used to prove the weak convergence.
A similar consideration was used in Li \cite{Li-averaging}, which was
inspired by a paper of Hairer and Pavliotis \cite{Hairer-Pavliotis}.
In the above-mentioned papers, the convergence is in probability; while
here we can only expect weak convergence. To prove the convergence, we
apply Stroock--Varadhan's martingale method
and Lemma~\ref{lem3}; see also Borodin and Freidlin \cite{Borodin-Freidlin95};
Papanicolaou, Stroock and Varadhan \cite
{Papanicolaou-Stroock-Varadhan73,Papanicolaou-Stroock-Varadhan77} where
the limit is given by a double integration in time. Our formulation for
the limit is in terms of space averaging. Finally, we use explicit
eigenfunctions of the Laplacian on $\mathit{SO}(n)$ to compute the limiting generator.

\begin{pf*}{Proof of Theorem \ref{thm1}}
We define a Markov generator $\bar\L$ on $\mathit{OM}$. If $F\dvtx \mathit{OM}\to\R$ is
bounded and Borel measurable and
$\{e_i\}$ is an orthonormal basis of $\R^n$,
we define
%
\begin{eqnarray}
\label{generator} \bar\L F&=&- \sum_{i=1}^n
\int_G (\nabla DF)_{u} \bigl(H_{u}(
ge_0), H_{u} ( e_i) \bigr) h_i
(g) \,dg
\nonumber
\\[-8pt]
\\[-8pt]
\nonumber
&&{} -\sum_{i=1}^n \int_G
(DF)_u (H_ue_i) L_{g\bar A }
h_i( g) \,dg,
\end{eqnarray}
where $h_i$ is the solution to the Poisson equation (\ref{Poisson}).
Since $(\tilde x_{t/\varepsilon}^\varepsilon)$
is tight by Lemma~\ref{lem2}, every sub-sequence of $(\tilde x_{t/
\varepsilon}^\varepsilon)$
has a sub-sequence that converges in distribution.
We will prove that the probability distributions\vspace*{-1pt} of $(\tilde x_{t/
\varepsilon}^\varepsilon)$ converge weakly
to the probability measure, $\bar P$, determined by $\bar\L$.
It is sufficient to prove that if $(\bar y_t)$ is a limit
of $(\tilde x_{t/\varepsilon}^\varepsilon)$, then
\[
F(\bar y_t)-F(u_0)-\int_0^t
\bar\L F(\bar y_s)\,ds
\]
is a martingale. Since the convergence is weak, and the Markov process
$(\tilde x_t^\varepsilon, g^\varepsilon_{t/\varepsilon})$ is not tight, we
do not
have a suitable filtration on $\Omega$ to work with.
We formulate the above convergence on the space of continuous
paths over $\mathit{OM}$ on a given time interval $[0,T]$.

Let $X_t$ be the coordinate process on the path space over $\mathit{OM}$, $\G
_t=\sigma\{(X_s)\dvtx 0\le s\le t\}$ and let
$\hat P_{\tilde x^\varepsilon}$ be the probability distribution\vspace*{-1pt} of
$(\tilde x_{t/\varepsilon}^\varepsilon)$ on the
path space over $\mathit{OM}$. By taking a subsequence if necessary, we may
assume that
$\{\hat P_{\tilde x^\varepsilon}\}$ converges to $\bar P$.

Let $F\dvtx \mathit{OM}\to\R$ be a smooth function with compact support.
We will prove that with respect to $\bar P$,
\[
\E \biggl\{F(X_t)-F(X_r)-\int_r^t
\bar\L F(X_s)\,ds \Big| \G_r \biggr\}=0.
\]
Since $\hat P_{\tilde x_\varepsilon}\to\bar P$ weakly, we only need to
prove that
for all bounded and continuous real value random variables $\xi$ that
are measurable with respect
to $\G_r$,
%
\begin{equation}
\label{proof2} \lim_{\varepsilon\to0}\int\xi\bigl(F(X_t)-F(X_r)
\bigr) \,d\hat P_{\tilde
x^\varepsilon} =\int \biggl(\xi\int_r^t
\bar\L F(X_s) \,ds \biggr) \,d\bar P.
\end{equation}
By formula (\ref{generator-1}) in the proof of Lemma~\ref{lem2}, for
$t\ge r$,
%
\begin{eqnarray}
\label{generator-2} &&F\bigl(\tilde x_{t/\varepsilon}^{\varepsilon}\bigr)-F\bigl(\tilde
x_{r/
{\varepsilon
}}^{\varepsilon}\bigr)\nonumber
\\
&&\qquad\sim-{\varepsilon} \sum_{i=1}^n\int
_{r/{\varepsilon}}^{t/
{\varepsilon}} (\nabla DF)_{ \tilde x_s^{\varepsilon}}
\bigl(H_{\tilde x_s^{\varepsilon}}\bigl( g_s^{\varepsilon} e_0
\bigr), H_{\tilde x_s^{\varepsilon}} ( e_i) \bigr) h_i
\bigl(g_s^{\varepsilon}\bigr)\,ds
\nonumber
\\[-8pt]
\\[-8pt]
\nonumber
&&\qquad\quad{}-{\varepsilon} \sum_{i=1}^n\int
_{r/{\varepsilon}}^{t/{\varepsilon
}} ( DF)_{\tilde x_s^{\varepsilon}} (H_{\tilde x_s^{\varepsilon}}
e_i ) L_{g_s^{\varepsilon} \bar A } h_i\bigl( g_s^{\varepsilon}
\bigr)\,ds
\\
&&\qquad\quad{}-\sqrt{\varepsilon} \sum_{i=1}^n\sum
_{k=1}^N \int_{r/{\varepsilon
}}^{t/{\varepsilon}}
( DF)_{\tilde x_{s}^{\varepsilon}} (H_{\tilde x_{s}^{\varepsilon}} e_i ) (Dh_i)_{( g_s^{\varepsilon})}
\bigl(g_s^{\varepsilon} A_k\bigr) \,dw_s^k.
\nonumber
\end{eqnarray}

Hence, up to a term of order $\varepsilon$,
\begin{eqnarray*}
&&\int\xi\bigl(F(X_t)-F(X_r)\bigr) \,d\hat
P_{\tilde x^\varepsilon}
\\
&&\qquad=O(\varepsilon) -\varepsilon\sum_{i=1}^n
\int \biggl(\xi\int_{r/
\varepsilon}^{t/\varepsilon} (\nabla
DF)_{ X_s} \bigl(H_{X_s}( G_s e_0),
H_{X_s} ( e_i) \bigr) h_i (G_s)\,ds
\biggr)\,d\hat P_{\tilde x^\varepsilon}
\\
&&\qquad\quad{}-\varepsilon\sum_{i=1}^n\int \biggl(\xi
\int_{r/\varepsilon
}^{t/
\varepsilon} ( DF)_{X_s}
(H_{X_s} e_i ) L_{G_s\bar A } h_i(
G_s)\,ds \biggr)\,d\hat P_{\tilde x^\varepsilon}. 
\end{eqnarray*}
We prove this by working with the original processes. Let $(\tilde
x_t^\varepsilon)$ denote a sub-sequence of the original sequence with
limit $(\bar y_s)$.
For each $i, l=1,\ldots, n$, let us define
\[
\beta_{li}(u)=(\nabla DF)_{u} \bigl(H_{u}(
e_l), H_{u} ( e_i) \bigr).
\]
By linearity of $H_u$ and $\nabla DF$,
\begin{eqnarray*}
&&(\nabla DF)_{u} \bigl(H_{u}( g e_0),
H_{u} e_i \bigr) h_i (g)
\\
&&\qquad=\sum_{l=1}^n(\nabla
DF)_{u} \bigl(H_{u}( e_l), H_{u}
( e_i) \bigr) \langle ge_0, e_l \rangle
h_i (g)
=\sum_{l=1}^n\beta_{li}(u)
\langle ge_0, e_l \rangle h_i (g),
\end{eqnarray*}
for each $i=1,\ldots, n$; and
\begin{eqnarray*}
&& -\varepsilon\int_{r/\varepsilon}^{t/\varepsilon} (\nabla
DF)_{ \tilde x^\varepsilon_s} \bigl(H_{ \tilde x^\varepsilon_s}\bigl( g^\varepsilon_s
e_0\bigr), H_{ \tilde x^\varepsilon(s )} ( e_i) \bigr)
h_i \bigl( g^\varepsilon_s\bigr) \,ds
\\
&&\qquad =-\varepsilon\sum_{l=1}^n \int
_{r/\varepsilon}^{t/\varepsilon} \beta_{li}\bigl(\tilde
x^\varepsilon_s \bigr) \bigl\langle g^\varepsilon_s
e_0, e_l \bigr\rangle h_i
\bigl(g^\varepsilon_s\bigr) \,ds
\\
&&\qquad =- \sum_{l=1}^n \int
_{r}^{t} \beta_{li} \bigl(\tilde
x^\varepsilon _{s/
\varepsilon} \bigr) \bigl\langle g^\varepsilon_{s/\varepsilon}
e_0, e_l \bigr\rangle h_i
\bigl(g^\varepsilon _{s/\varepsilon}\bigr) \,ds.
\end{eqnarray*}
We observe that $( g^\varepsilon_{s\varepsilon} )$ satisfies the equation
$dg_t=\sum_k g_tA_k \circ dw_t^k$ with initial value the identity element.
The solution stays in the connected component $\mathit{SO}(n)$. It is ergodic
with the normalized Haar measure $dg$ on $\mathit{SO}(n)$ as its
invariant measure and it satisfies the Birkhoff ergodic theorem; see
Example~\ref{ex1}.
By Lemma~\ref{lem2}, $(\tilde x^\varepsilon_{s/\varepsilon} )$ is tight,
and equi-uniformly H\"older continuous on $[0,T]$.
In Assumption~\ref{assum}, we take $z_t^\varepsilon=g_t^\varepsilon$,
$d\mu_\varepsilon=dg$, $y^\varepsilon_t =\tilde x^\varepsilon_t$ and check that
conditions (1)--(4) are satisfied. In Lemma~\ref{lem3}, we take
$G(u,g)=\sum_{l=1}^n\beta_{li}(u)  \langle ge_0, e_l
\rangle h_i (g)$.
Since the functions $h_i\dvtx G\to\R$ are smooth and $G$ is compact, also
$\beta_{li}$ are smooth and bounded
by construction, we may apply Lemma~\ref{lem3}.
If $\phi$ is a bounded real valued continuous function on $C([0, r];
\mathit{OM})$, let $\xi=\phi(\tilde x^\varepsilon_{u/\varepsilon}, 0\le u \le r)$.
Then
\begin{eqnarray*}
&&\lim_{\varepsilon\to0}\E \Biggl( \xi\sum_{l=1}^n
\int_{r}^{t} \beta_{li}\bigl(\tilde
x^\varepsilon_{s/\varepsilon} \bigr) \bigl\langle g^\varepsilon
_{s/\varepsilon} e_0, e_l \bigr\rangle h_i
\bigl(g^\varepsilon_{s/
\varepsilon
}\bigr) \,ds \Biggr)
\\
&&\qquad= \sum_{l=1}^n \E \biggl( \xi\int
_{r}^{t} \beta_{li}(\bar
y_s )\,ds \biggr) \int_G \langle
ge_0, e_l \rangle h_i (g) \,dg
\\
&&\qquad= \sum_{l=1}^n \E \biggl( \xi\int
_{r}^{t} \nabla DF_{\bar y_s}
\bigl(H_{\bar y_s}( e_l), H_{\bar y_s} ( e_i)
\bigr) \biggr) \int_G \langle ge_0,
e_l\rangle h_i (g) \,dg
\\
&&\qquad= \sum_{l=1}^n \E \biggl( \xi\int
_{r}^{t} \int_G \nabla
DF_{\bar y_s} \bigl(H_{\bar y_s}( ge_0), H_{\bar y_s}
( e_i) \bigr) h_i (g) \,dg \biggr).
\end{eqnarray*}
%

By the same reasoning, we also have
\begin{eqnarray*}
&&\lim_{\varepsilon\to0}{\varepsilon} \E \biggl( \xi\int
_{r/
{\varepsilon
}}^{t/{\varepsilon}} ( DF)_{\tilde x_s^{\varepsilon}} (H_{\tilde
x_s^{\varepsilon}}
e_i ) L_{g_s^{\varepsilon} \bar A } h_i\bigl( g_s^{\varepsilon
}
\bigr)\,ds \biggr)
\\
&&\qquad= \E \biggl(\xi\int_r^t (DF)_{\bar y_s}
(H_{\bar y_s}e_i)\,ds \int_G
L_{g\bar A } h_i( g) \,dg \biggr).
\end{eqnarray*}
We have proved (\ref{proof2}). Since every sub-sequence of $\hat
P_{\tilde x^\varepsilon}$
has a sub-sequence that converges to the same limit, we have proved
$\hat P_{\tilde x^\varepsilon} \to\bar P$ weakly.

Finally, we compute the limiting Markov generator $\bar\L$. We observe
that there is a family of eigenfunctions
of the Laplacian on $G$ with eigenvalue $-{n-1\over2}$.
Indeed, since $\sum_{k=1}^{n(n-1)/2} (A_k)^2=-{n-1\over2}I$,
\begin{eqnarray*}
\sum_{k=1}^{n(n-1)/2}L_{gA_k}L_{gA_k}
\biggl(-{4\over
n-1}\langle ge_0, e_i\rangle
\biggr) &=&-{4\over n-1}\sum_{k=1}^{n(n-1)/2}
\bigl\langle g(A_k)^2e_0, e_i
\bigr\rangle
\\
&=&2\langle ge_0, e_i\rangle.
\end{eqnarray*}
%
Thus,
\[
h_i=-{4\over n-1}\langle ge_0,
e_i\rangle
\]
is the solution to the Poisson equation (\ref{Poisson}):
\[
\L_G h_i=\langle ge_0, e_i
\rangle\qquad \mbox{where } \L_G ={1\over
2} \sum
_{k=1}^{n(n-1)/2}L_{gA_k}L_{gA_k}.
\]
We compute the second integral in (\ref{generator}). Since $L_{g\bar A
}h_i=-{4\over n-1} \langle g\bar A e_0, e_i\rangle $, we have
\begin{eqnarray*}
&&\sum_{i=1}^n \int_G(DF)_u(H_ue_i)L_{g\bar A}h_i(g)\,dg
\\
&&\qquad=-{4\over n-1}\int_G (DF)_u (
H_u g\bar A e_0)\,dg
\\
&&\qquad=-{4\over n-1} (DF)_u \biggl( H_u \biggl(
\int_G g\bar A e_0\,dg \biggr) \biggr)=0.
\end{eqnarray*}
Consequently,
\begin{eqnarray*}
\bar\L F&=&- \sum_{i=1}^n\int
_G (\nabla DF)_{u} \bigl(H_{u}(
ge_0), H_{u} ( e_i) \bigr) h_i
(g) \,dg
\\
&=&- \sum_{i,j=1}^n\int
_G (\nabla DF)_{u} \bigl(H_{u}(
e_j), H_{u} ( e_i) \bigr) \langle
ge_0, e_j\rangle h_i (g) \,dg.
\end{eqnarray*}
In the last step, we use the fact that $H_u(\cdot)$ is linear and that
$\{e_i\}$ is an o.n.b. of $\R^n$.
Let us define
\begin{eqnarray*}
a_{i,j}(e_0)&=&-\int_G\langle
ge_0, e_j\rangle h_i(g)\,dg
\\
&=&{4\over n-1}\int_G\langle
ge_0, e_j\rangle \langle ge_0,
e_i\rangle \,dg.
\end{eqnarray*}
Then
%
\begin{equation}
\label{generator-10} \bar\L F=- \sum_{i,j=1}^n
a_{i,j} (\nabla DF)_{u} \bigl(H_{u}(
e_j), H_{u} ( e_i) \bigr).
\end{equation}

To further identify the limit, we first prove that $a_{i,j}(e_0)$ is
independent of $e_0$. Recall that
$G$ acts transitively on the unit sphere of $\R^n$.
Let $e_0'\in\R^n$ we take $O $ such that $Oe_0'=e_0$. By the right
invariant property of the Haar measure,
\[
\int_G\bigl\langle ge_0',
e_j\bigr\rangle \bigl\langle ge_0',
e_i\bigr\rangle \,dg=\int_G\langle
gOe_0, e_j\rangle \langle gOe_0,
e_i\rangle \,dg =\int_G\langle
ge_0, e_j\rangle \langle ge_0,
e_i\rangle \,dg.
\]
We first compute the case of $i\neq j$ and $n=2$: 
%
\[
a_{1,2}(e_1)=\int_{\mathit{SO}(2)}\langle
ge_1, e_1\rangle \langle ge_1,
e_2\rangle \,dg =-\int_0^{2\pi} \cos(
\theta)\sin(\theta) \,d\theta=0.
\]

If $n>2$, for any $i\neq j$, there is an orientation preserving
rotation matrix $O$ such that
$Oe_i=-e_i$ and $Oe_j=e_j$. For example, if $i=1, j=2$, we take
$O=(-e_1, e_2, -e_3, e_4, \ldots, e_n)$. So
\begin{eqnarray*}
\int_G\langle ge_0, e_j\rangle
\langle ge_0, e_i\rangle \,dg&=&-\int_G
\langle ge_0, Oe_j\rangle \langle ge_0,
Oe_i\rangle \,dg
\\
&=&-\int_G\langle ge_0, e_j
\rangle \langle ge_0, e_i\rangle \,dg.
\end{eqnarray*}

Thus, $a_{i,j}=0$ if $i\neq j$.
Let
\[
C_i=\int_G \langle ge_0,
e_i\rangle ^2 \,dg.
\]
For $i=1, \ldots, n$, $C_i=\int_G \langle ge_0, e_i\rangle ^2 \,dg$ is
independent of
$i$ and
\[
\int_G\sum_{i=1}^n
\langle ge_0, e_i\rangle ^2 \,dg=1
\]
and consequently $ C_i={1\over n}$. The nonzero values of $(a_{i,j})$ are
\[
a_{i,i}=-\int_G \langle ge_0,
e_i\rangle h_i(g)\,dg={ 4\over n-1 } \int
_G\langle ge_0, e_i\rangle
^2 \,dg ={4\over(n-1)n}.
\]
By the definition, $\Delta_H F(u)=\sum_{i=1}^n L_{H(e_i)}L_{H(e_i)}F$.
Since $\nabla$ is the canonical flat connection,
$\nabla_{H(e_i)}H(e_i)=0$. See the paragraph before equation (\ref
{generator-1}).
By (\ref{generator-10}), we see that
\begin{eqnarray*}
\bar\L F (u)&=&- \sum_{i,j=1}^n
a_{i,j} (\nabla DF)_{u} \bigl(H_{u}(
e_j), H_{u} ( e_i) \bigr)
\\
&=&{4\over(n-1)n}\sum_{i=1}^n (
\nabla DF)_u \bigl(H_{u}( e_i),
H_{u} ( e_i) \bigr)
\\
&=&{4\over(n-1)n}\Delta_H F(u).
\end{eqnarray*}
We conclude that $(\tilde x^\varepsilon_{t/\varepsilon})$ is a diffusion
process with infinitesimal generator ${4\over(n-1)n}\Delta_H$.
Since $(x^\varepsilon_{t/\varepsilon})$ is the projection of $(\tilde
x^\varepsilon_{t/\varepsilon})$ it is also convergent.
The operators $\Delta_H$ and $\Delta$ are intertwined by $\pi$; for $f\dvtx
M\to\R$ smooth,
$(\Delta_H f)\circ\pi=\Delta(f\circ\pi)$. See, for example,
Theorem~4C of Chapter II in Elworthy \cite{Elworthy-SaintFlour} and also
Elworthy, Le Jan and Li \cite{Elworthy-LeJan-Li-book-2};
$\Delta_H$ is cohesive and a horizontal operator in the terminology of
\cite{Elworthy-LeJan-Li-book-2} and is the horizontal lift
of $\Delta$.
We see that\vspace*{-1pt} $(x^\varepsilon_{t/\varepsilon})$ converges to a process
with generator
${4\over(n-1)n}\Delta$ where $\Delta$ is the Laplacian on the
Riemannian manifold $M$. We have completed the proof of Theorem~\ref{thm1}.
\end{pf*}

\section*{Acknowledgements} It is a pleasure to thank D.~Bakry,
K.~D. Elworthy, M.~Hairer, M.~Ledoux,
Y.~Maeda, J.~Norris and S.~Rosenberg for helpful discussions. I would
also like to thank the referees
for helpful comments.

%
%





\printaddresses
\end{document}